\documentclass[11pt]{amsart}
\usepackage{amssymb,amsxtra,amsmath,amsfonts}
\usepackage{verbatim}
\usepackage{cite}
\newtheorem{theorem}{Theorem}

\newtheorem{proposition}[theorem]{Proposition}
\newtheorem{corollary}[theorem]{Corollary}

\theoremstyle{definition}
\newtheorem{definition}[theorem]{Definition}
\newtheorem{example}[theorem]{Example}

\theoremstyle{remark}
\newtheorem{remark}[theorem]{Remark}

\numberwithin{equation}{section}
\numberwithin{theorem}{section}

\newcommand\thref{Theorem \ref}

\newcommand\prref{Proposition \ref}
\newcommand\coref{Corollary \ref}

\newcommand\exref{Example \ref}

\newcommand\seref{Section \ref}
\renewcommand{\comment}[1]{}
\def\ii{\mathrm{i}} %{\sqrt{-1}}
\def\vac{\boldsymbol{1}}
\def\CC{\mathbb{C}}

\def\ZZ{\mathbb{Z}}
\def\N{\mathcal{N}}
\def\W{\mathcal{W}}
\def\ph{\varphi}
\def\om{\omega}
\def\si{\sigma}
\def\lieh{\mathfrak{h}}
\def\hh{\hat{\lieh}}
\def\hhp{\hh_\ph}
\def\liea{\mathfrak{a}}

\def\caph{(C_{\liea})_\ph}
\def\al{\alpha}
\def\be{\beta}
\def\ga{\gamma}
\def\de{\delta}
\def\la{\lambda}
\def\ze{\zeta}

\newcommand{\h}{\mathfrak{h}}
\newcommand{\nop}[1]{\textnormal{{:}}#1\textnormal{{:}}}
\newcommand{\nopm}[1]{{}^\circ_\circ \hspace{-.7sp}#1 \hspace{-.5sp} {}^\circ_\circ}
\DeclareMathOperator\re{Re}
\DeclareMathOperator\im{Im}
\DeclareMathOperator\sdim{sdim}
\DeclareMathOperator\LF{LFie}

\DeclareMathOperator\Span{span}
\DeclareMathOperator\ind{Ind}
\DeclareMathOperator\aut{Aut}
\DeclareMathOperator\Aut{Aut}

\DeclareMathOperator\der{Der}
\DeclareMathOperator\res{Res}

\DeclareMathOperator\str{str}

\begin{document}

\title{Twisted logarithmic modules of free field algebras}
\author{Bojko Bakalov}
\address{Department of Mathematics, North Carolina State University, 
Raleigh, NC 27695, USA}
\email{bojko\_bakalov@ncsu.edu; smsulli4@ncsu.edu}

\author{McKay Sullivan}
%\address{Department of Mathematics, North Carolina State University, Raleigh, NC 27695, USA}
%\email{smsulli4@ncsu.edu}

\begin{abstract}
Given a non-semisimple automorphism $\ph$ of a vertex algebra $V$, the fields in a $\ph$-twisted $V$-module involve the logarithm of the formal variable, and the action of the Virasoro operator $L_0$ on such module is not semisimple. We construct examples of such modules and realize them explicitly as Fock spaces when $V$ is generated by free fields. Specifically, we consider the cases of symplectic fermions (odd superbosons), free fermions, and $\be\ga$-system (even superfermions).
In each case, we determine the action of the Virasoro algebra.
\end{abstract}

\thanks{The first author is supported in part by a Simons Foundation grant 279074} 

\thanks{Published in \textit{J. Math. Phys.} \textbf{57}, 061701 (2016)}

\date{April 1, 2016; Revised June 16, 2016}



 \maketitle

%\tableofcontents

%%%%%%%%%%%%%%%%%%%%%%%%%%%%%%%%%%%%%%%%%%%%%%%%%
%%%%%%%%%%%%%%%%%%%%%%%%%%%%%%%%%%%%%%%%%%%%%%%%%
\section{Introduction}
%%%%%%%%%%%%%%%%%%%%%%%%%%%%%%%%%%%%%%%%%%%%%%%%%
%%%%%%%%%%%%%%%%%%%%%%%%%%%%%%%%%%%%%%%%%%%%%%%%%

The notion of a \textit{vertex algebra} introduced by Borcherds \cite{Bo} provides a rigorous algebraic description of two-dimensional chiral conformal field theory \cite{BPZ, Go, DMS}, and is a powerful tool for studying representations of infinite-dimensional Lie algebras \cite{K1,KRR}. We will assume that the reader is familiar with the theory of vertex algebras as presented in \cite{K2} (see also \cite{FLM,FB,LL,KRR} for other sources).
Given an automorphism $\varphi$ of a  vertex algebra $V$, one considers the so-called \emph{$\varphi$-twisted $V$-modules} \cite{FLM,D,FFR,BK}, which are useful for constructing modules of the \emph{orbifold} subalgebra consisting of elements fixed by $\ph$ (see e.g.\ \cite{DVVV,KT,DLM}).
More recently, the notion of a $\varphi$-twisted module was generalized to the case when $\varphi$ is not semisimple \cite{H,Bak:2015}.
This generalization was motivated by logarithmic conformal field theory (see e.g.\ \cite{Kau1,AM,CR}) and applications to Gromov--Witten theory (cf.\ \cite{DZ,MT,BM,BW}).
The main feature of such modules is that the twisted fields involve the logarithm of the formal variable; for this reason we also call them \emph{twisted logarithmic modules}.

More precisely, let $z$ and $\ze$ be two independent formal variables. If we think of $\ze$ as $\log z$, then the derivatives with respect to $z$ and $\ze$ will become
\begin{equation}\label{dzdze}
D_z = \partial_z + z^{-1}\partial_\ze, \qquad D_\zeta = z \partial_z + \partial_\ze.
\end{equation}
For a vector space $W$ over $\CC$, a \emph{logarithmic (quantum) field} on $W$ is a linear map from $W$ to the space of formal series of the form
\begin{equation*}
\sum_{m\in S} \sum_{i=0}^\infty w_{i,m}(\ze)z^{i+m}, \qquad w_{i,m}(\ze) \in W[\ze],
\end{equation*}
for some finite subset $S$ of $\CC$. The space of all logarithmic fields is denoted $\LF(W)$.
Given a vertex algebra $V$ and an automorphism $\ph$ of $V$, a \emph{$\ph$-twisted $V$-module} is a vector space $W$, equipped with a linear map $Y\colon V \to \LF(W)$
satisfying certain axioms (see \cite{Bak:2015} for full details). 

The fields $Y(a)$ are usually written as $Y(a,z)$ for $a\in V$ (even though they also depend on $\ze$).
One of the axioms is the \emph{$\ph$-equivariance}
\begin{equation*}%\label{twlog2}
Y(\ph a,z) = e^{2\pi\ii D_\ze}Y(a,z), \qquad a\in V.
\end{equation*}
As in \cite{Bak:2015}, we will assume that $\ph= \sigma e^{-2\pi\ii \N}$, where $\sigma \in \aut(V)$, $\N \in  \der(V)$, $\sigma$ and $\N$ commute, $\sigma$ is semisimple, and $\N$ is locally nilpotent (i.e., nilpotent on every $a\in V$).
We denote the eigenspaces of $\sigma$ by 
\begin{equation*}
V_\al = \{a \in V \,|\, \sigma a = e^{-2\pi\ii \al} a\}, \qquad \al \in \CC/\ZZ.
\end{equation*}
The $\ph$-equivariance implies that 
\begin{equation}\label{tfieldx}
X(a,z) = Y(e^{\zeta \N}a,z) = Y(a,z)\big|_{\zeta = 0}
\end{equation}
is independent of $\ze$, and the exponents of $z$ in $X(a,z)$ belong to $-\al$ for $a\in V_\al$.
%Whenever $a \in V_\al$ we will let $\al_0$ be the coset representative of $\al$ such that $-1<\re \al_0 \leq 0$.
For $m \in \al$, the \emph{$(m+\N)$-th mode} of $a\in V_\al$ is defined as
\begin{equation}\label{tmodes}
a_{(m+\N)} = \res_z z^m X(a,z),
\end{equation}
where, as usual, $\res_z$ denotes the coefficient of $z^{-1}$.
Then
\begin{equation}\label{tfield}
Y(a,z) = X(e^{-\zeta \N}a,z) = \sum_{m \in \al}(z^{-m-\N-1}a)_{(m+\N)},
\end{equation}
where we use the notation $z^{-\N}=e^{-\ze \N}$.
Notice that $Y(a,z)$ is a polynomial of $\ze$, since $\N$ is nilpotent on $a$.
The paper \cite{Bak:2015} develops the theory of twisted logarithmic modules and, in particular, proves a \emph{Borcherds identity} and \emph{commutator formula}
for the modes \eqref{tmodes}. 
%and the following \emph{commutator formula}:
%\begin{equation}\label{tcomm}
%[a_{(m+\N)},b_{(k+\N)}] =\sum_{j=0}^\infty\Big(\Big({m+\N \choose j }a\Big)_{(j)}b\Big)_{(m+k-j+\N)},
%\end{equation}
%where $a\in V_\al$, $b\in V_\be$, $m\in\al$, $k\in\be$.
It contains examples of modules in the cases when $V$ is an affine vertex algebra or a Heisenberg vertex algebra. 
The latter is also known as the \emph{free boson} algebra (see e.g.\ \cite{K2}).

In the present paper, we extend the results of \cite{Bak:2015} to the case of \emph{symplectic fermions}, which are odd super-analogs of the free bosons.
Historically, the symplectic fermions provided the first example of a logarithmic conformal field theory, due to Kausch \cite{Kau1,Kau2}.
An orbifold of the symplectic fermion algebra $SF$ 
under an automorphism of order $2$ has the important properties of being $C_2$-cofinite but not rational \cite{Abe:2007}.
More recently, the subalgebras of $SF$ %the symplectic fermion vertex algebra 
known as the triplet and singlet algebras have generated considerable interest (see \cite{Kau3,AM}).
Other orbifolds of $SF$ give rise to interesting $\W$-algebras \cite{CL}.

We also consider the examples of \emph{free superfermions} (cf.\ \cite{K2}), which include in particular the \emph{free fermions} and the symplectic bosons
(also known as the bosonic ghost system or \emph{$\be\ga$-system}). 
The $\be\ga$-system provides another interesting model of logarithmic conformal field theory \cite{RW}.
Many important algebras have \emph{free-field realizations} by superfermions.
In particular, these include: affine Kac--Moody algebras \cite{F,KP,W,FF1,FF4}, affine Lie superalgebras \cite{KW1}, toroidal Lie algebras \cite{JM,JMX}, $\W$-algebras \cite{FF2,FF3,KRW}, superconformal algebras \cite{KRW,KW2}, the $\W_{1+\infty}$-algebra and its subalgebras \cite{KR,KWY,L}. 
If $V$ is one of the free-field algebras, $\ph\in\Aut(V)$, and $A\subset V$ is a subalgebra such that $\ph(A)\subset A$, then any $\ph$-twisted $V$-module gives rise to a $\ph$-twisted $A$-module by restriction. Moreover, such $A$-modules are untwisted if $\ph$ acts as the identity on $A$ (i.e., if $A\subset V^\ph$). Thus, we expect the twisted modules constructed in this paper to be useful for studying modules over subalgebras of free-field algebras.

The paper is organized as follows. In \seref{sfsb}, following \cite{K2}, we review the definition of free superbosons, which include the free bosons in the even case and the symplectic fermions $SF$ in the odd case. Then we outline the construction of $\ph$-twisted $SF$-modules as modules over $\ph$-twisted affine Lie superalgebras, similarly to the even case considered in \cite{Bak:2015}. In order to make the construction explicit, we need to solve a linear algebra problem, which we do in \seref{sautbil}. Using that, in \seref{ssfm}, we realize explicitly all highest-weight $\ph$-twisted $SF$-modules as certain Fock spaces and we determine the action of the Virasoro algebra on them. 
In particular, we confirm that the action of the Virasoro operator $L_0$ is not semisimple
(cf.\ \cite{Kau1,AM,CR}). In \seref{ssferm}, we study twisted logarithmic modules of free superfermions and determine the action of the Virasoro algebra. In the final Sections \ref{sec:ff} and \ref{sec:bgs}, we realize these modules explicitly as Fock spaces in the odd case (free fermions)
and the even case ($\be\ga$-system), respectively.

%%%%%%%%%%%%%%%%%%%%%%%%%%%%%%%%%%%%%%%%%%%%%%%%%
%%%%%%%%%%%%%%%%%%%%%%%%%%%%%%%%%%%%%%%%%%%%%%%%%
\section{Free superbosons}\label{sfsb}
%%%%%%%%%%%%%%%%%%%%%%%%%%%%%%%%%%%%%%%%%%%%%%%%%
%%%%%%%%%%%%%%%%%%%%%%%%%%%%%%%%%%%%%%%%%%%%%%%%%

In this section, we review the definition of free superbosons, following \cite{K2}. 
Let $\lieh=\lieh_{\bar0}\oplus\lieh_{\bar1}$ be an abelian Lie superalgebra with $\dim \lieh = d< \infty$. Let $(\cdot|\cdot)$ be a non-degenerate even supersymmetric bilinear form on $\lieh$. Thus, $(b|a) = (-1)^{p(a)p(b)}(a|b)$ and $(\lieh_{\bar0}|\lieh_{\bar1})=0$, where $p(a)$ denotes the parity of $a$.
Consider the Lie superalgebra given by the affinization $\hh=\h[t,t^{-1}]\oplus \CC K$  with commutation relations
\begin{equation}\label{sbbrackets}
[at^m,bt^n] = m \delta_{m,-n}(a|b)K, \qquad [\hh,K] = 0 \qquad (m,n \in \ZZ),
\end{equation}
and $p(at^m)=p(a)$, $p(K) = \bar{0}$.
We will use the notation $a_{(m)}=at^m$. The \emph{free superbosons}
\begin{equation*}
a(z) = \sum_{m\in \ZZ}a_{(m)}z^{-m-1}, \qquad a \in \lieh,
\end{equation*}
have OPEs given by 
\begin{equation*}%\label{sbOPE}
a(z)b(w) \sim \frac{(a|b)K}{(z-w)^2} \,.
\end{equation*}
The (generalized) Verma module 
\begin{equation*}%\label{bverma1}
V = \ind_{\h[t]\oplus \CC K}^{\hat{\h}} \CC
\end{equation*}
is constructed by letting $\h[t]$ act trivially on $\CC$ and $K$ act as $1$. Then $V$ has the structure of a vertex algebra called the \textit{free superboson algebra} and denoted $B^1(\h)$.
The commutator \eqref{sbbrackets} is equivalent to the following $n$-th products:
\begin{align}\label{nthprodb}
a_{(0)}b &= 0, & a_{(1)}b &= (a|b)\vac, & a_{(j)}b &= 0 &&  (j \geq 2)
\end{align} 
for $a,b \in \h$, where $\vac$ is the vacuum vector in $B^1(\h)$.

Let $\ph$ be an even automorphism of $\h$ such that $(\cdot|\cdot)$ is $\ph$-invariant. As before, we write $\ph=\si e^{-2 \pi\ii\N},$ and denote the eigenspaces of $\sigma$ by 
\begin{equation*}
\h_\al = \{a \in \h \,|\, \sigma a = e^{-2\pi\ii \al} a\}, \qquad \al \in \CC/\ZZ.
\end{equation*}

\begin{definition}[cf. \cite{Bak:2015}]\label{def:twaff}
The \emph{$\ph$-twisted affinization} $\hhp$ is the Lie superalgebra spanned by an even central element $K$ and elements $a_{(m+\N)}=at^m$ $(a \in \h_\al, \, m \in \al)$ with parity $p(a_{(m+\N)}) = p(a)$, and the Lie superbracket
\begin{equation}\label{waffbrack}
[a_{(m+\N)},b_{(n+\N)}]=\delta_{m,-n}((m+\N)a|b)K
\end{equation}
for $a \in \h_\al$, $b \in \h_\beta$, $m \in \al,$ $n\in \beta$.
\end{definition}

An $\hhp$-module $W$ is called \textit{restricted} if for every $a \in \lieh_\al$, $m \in \al$, $v \in W$, there is an integer $L$ such that $(at^{m+i})v =0$ for all $i \in \ZZ,$ $i \geq L$. We note that every highest weight $\hhp$-module is restricted (see \cite{K1}).
The automorphism $\ph$ naturally induces automorphisms of $\hat{\h}$ and $B^1(\h)$, which we will denote again by $\ph$.  Then every $\ph$-twisted $B^1(\lieh)$-module is a restricted $\hhp$-module and, conversely, every restricted $\hhp$-module uniquely extends to a $\ph$-twisted $B^1(\lieh)$-module \cite[Theorem 6.3]{Bak:2015}. 

We split $\CC$ as a disjoint union of subsets $\CC^+$, $\CC^-=-\CC^+$ and $\{0\}$ where 
\begin{equation}\label{cplus}
\CC^+ = \{\gamma \in \CC \,|\, \re \gamma > 0 \} \cup \{ \gamma \in \CC \,|\, \re \gamma = 0, \, \im \gamma > 0\}. 
\end{equation}
Then the $\ph$-twisted affinization $\hat{\h}_\ph$ has a triangular decomposition 
\begin{equation}\label{wtriangle}
\hat{\h}_\ph = \hat{\h}_\ph^- \oplus \hat{\h}_\ph^0 \oplus \hat{\h}_\ph^+,
\end{equation}
 where
\begin{equation*}
\hat{\h}^\pm_\ph = \Span\{at^m \,|\, a \in \h_\al, \, \al \in \CC/\ZZ, \, m \in \al \cap \CC^\pm\}
\end{equation*}
and
\begin{equation*}
\hat{\h}_\ph^0 = \Span\{at^0 \,|\, a \in \h_0\}\oplus \CC K.
\end{equation*}
Starting from an $\hat{\h}^0_\ph$-module $R$ with $K=I$, the (generalized) \emph{Verma module} is defined by
\begin{equation*}%\label{bverma}
M_\ph(R) = \ind_{\hat{\h}_\ph^+\oplus \hat{\h}_\ph^0}^{\hat{\h}_\ph}R  ,
\end{equation*}
where $\hat{\h}_\ph^+$ acts trivially on $R$. These are $\ph$-twisted $B^1(\lieh)$-modules. 
In order to describe them more explicitly, first we will obtain canonical forms for all automorphisms $\ph$ of $\lieh$ preserving $(\cdot | \cdot )$.

%%%%%%%%%%%%%%%%%%%%%%%%%%%%%%%%%%%%%%%%%%%%%%%%%
%%%%%%%%%%%%%%%%%%%%%%%%%%%%%%%%%%%%%%%%%%%%%%%%%
\section{Automorphisms preserving a bilinear form}\label{sautbil}
%%%%%%%%%%%%%%%%%%%%%%%%%%%%%%%%%%%%%%%%%%%%%%%%%
%%%%%%%%%%%%%%%%%%%%%%%%%%%%%%%%%%%%%%%%%%%%%%%%%
In this section, we study automorphisms $\ph$ of a finite-dimensional vector space $V$ preserving a nondegenerate bilinear form $(\cdot | \cdot )$. Recall that $\ph$ can be written uniquely as $\ph = \sigma e^{-2 \pi\ii \N}$ where $\sigma$ is invertible and semisimple, $\N$ is nilpotent, and $\sigma$ and $\N$ commute. The $\ph$-invariance of $(\cdot | \cdot )$ is equivalent to:
\begin{equation}\label{inv}
(\sigma a | \sigma b) = (a | b), \qquad (\N a | b) + (a | \N b) = 0
\end{equation}
for all $a,b\in V$. We will assume that the bilinear form $(\cdot | \cdot )$ is either symmetric or skew-symmetric, and will consider these cases separately.

%%%%%%%%%%%%%%%%%%%%%%%%%%%%%%%%%%%%%%%%%%%%%%%%%
\subsection{The symmetric case}
%%%%%%%%%%%%%%%%%%%%%%%%%%%%%%%%%%%%%%%%%%%%%%%%%
The case when $(\cdot | \cdot )$ is symmetric was investigated previously in \cite[Section 6]{Bak:2015}. The classification of all $\si$ and $\N$ satisfying \eqref{inv} can be deduced from the well-known description of the canonical Jordan forms of orthogonal and skew-symmetric matrices over $\CC$ (see \cite{Ga:1959,HM:1999}). We include it here for completeness.
In the following examples, $V$ is a vector space with a basis $\{v_1,\ldots,v_d\}$ such that $(v_i|v_j)=\de_{i+j,d+1}$ for all $i,j$, and $\la=e^{-2\pi\ii\al_0}$ for some 
$\al_0 \in \CC$ such that $-1 < \re \al_0 \leq 0$.

\begin{example}[$d = 2\ell$]\label{ex:symm1}
\begin{align*}
\sigma v_i & = \begin{cases}
\lambda v_i, & 1\leq i \leq \ell,\\
\lambda^{-1}v_i & \ell+1 \leq i \leq 2\ell,
\end{cases} \qquad
\N v_i = \begin{cases}
v_{i+1}, & 1 \leq i \leq \ell-1,\\
-v_{i+1}, & \ell + 1 \leq i \leq 2\ell -1, \\
0, & i = \ell,2\ell.
\end{cases}
\end{align*}
The symmetry $v_i \mapsto (-1)^iv_{\ell+i}$, $v_{\ell+i} \mapsto (-1)^{i+\ell+1}v_i$ $(1\leq i \leq \ell)$ allows us to switch $\la$ with $\la^{-1}$ and assume that $\al_0 \in \CC^- \cup \{0\}$. 
If we write $\la^{-1}=e^{-2\pi\ii\be_0}$ where $-1 < \re \be_0 \leq 0$, then $\be_0 = -\al_0$ or $-\al_0-1$. Thus, after switching the roles of $\al_0$ and $\be_0$ if necessary, we may always assume that $-\frac{1}{2}\leq \re \al_0 \leq 0$ and $\im \al_0 \geq 0$ when $\re \al_0 =-\frac{1}{2}$.
\end{example}

\begin{example}[$d=2\ell-1$ and $\la=\pm 1$]\label{ex:symm2}
\begin{equation*}
\sigma v_i = \lambda v_i, \ \ 1 \leq i \leq 2\ell-1, \qquad \N v_i = \begin{cases}
(-1)^{i+1}v_{i+1}, &1 \leq i \leq 2\ell-2,\\
0, & i = 2\ell-1
\end{cases}
\end{equation*}
Since $\la=\pm 1$, it follows that $\al_0 =0$ or $-1/2$. 
\end{example}

\begin{remark}\label{rem:symm2}
After rescaling the basis vectors in \exref{ex:symm2}, the operator $\N$ can be written alternatively in the form
\begin{equation*}
\N v_i =  \begin{cases}
v_{i+1}, & 1 \leq i \leq \ell-1,\\
-v_{i+1}, & \ell \leq i \leq 2\ell -2, \\
0, & i = 2\ell-1,
\end{cases}
\end{equation*}
which more clearly shows the strong relationship between the symmetric and skew-symmetric cases 
(cf.\ \exref{ex:skew2} below).
\end{remark}

\begin{proposition}[\hspace{1sp}\cite{Bak:2015}]\label{prop:symm}
Let\/ $V$ be a finite-dimensional vector space, equipped with a nondegenerate symmetric bilinear form\/ $(\cdot|\cdot)$ and with commuting linear operators\/ $\sigma$, $\N$ satisfying \eqref{inv}, such that\/ $\sigma$ is invertible and semisimple, and\/ $\N$ is nilpotent. Then\/ $V$ is an orthogonal direct sum of subspaces that are as in Examples \ref{ex:symm1} and \ref{ex:symm2}.
\end{proposition}

%%%%%%%%%%%%%%%%%%%%%%%%%%%%%%%%%%%%%%%%%%%%%%%%%
\subsection{The skew-symmetric case}\label{subsec:skclass}
%%%%%%%%%%%%%%%%%%%%%%%%%%%%%%%%%%%%%%%%%%%%%%%%%
In the case when $(\cdot | \cdot )$ is a nondegenerate skew-symmetric bilinear form, we were unable to locate in the literature an analogous classification of symplectic matrices.
Below we present two explicit examples of linear operators $\sigma$ and $\N$ satisfying \eqref{inv}.
Pick a basis $\{v_1,\ldots, v_{2\ell}\}$ for $V$ such that 
\begin{equation}\label{skewbas}
(v_i|v_j) = \delta_{i+j,2\ell + 1} = -(v_j|v_i), \qquad 
1 \leq i \leq j \leq 2\ell,
\end{equation}
and let $\lambda = e^{-2 \pi\ii\al_0}$ as before.

\begin{example}[$\ell$ is odd, or $\ell$ is even and $\lambda \neq \pm 1$]\label{ex:skew1}

\begin{align*}
\sigma v_i & =
\begin{cases}
  \lambda v_i, &1 \leq i \leq \ell,\\
  \lambda^{-1} v_i, &\ell +1 \leq i \leq 2 \ell,\\
\end{cases}
&
\mathcal{N} v_i  &=
  \begin{cases} 
      v_{i+1}, &1 \leq i \leq \ell -1, \\
     -v_{i+1}, &\ell+1 \leq i \leq 2\ell -1, \\
       0, & i = \ell, 2\ell.
  \end{cases}
\end{align*}
As in \exref{ex:symm1},
the symmetry $v_i \mapsto (-1)^{i}v_{\ell+i},$ $ v_{\ell+i}\mapsto (-1)^{\ell+i} v_i$ $(1 \leq i \leq \ell)$ allows us to assume that $\al_0 \in \CC^- \cup \{0\}$,
$-\frac{1}{2}\leq \re \al_0 \leq 0$, and $\im \al_0 \geq 0$ when $\re \al_0 =-\frac{1}{2}$.
\end{example}

We have omitted the case when $\ell$ is even and $\lambda=\pm 1$ in \exref{ex:skew1}, because it can be rewritten as an orthogonal direct sum of two copies of the following example.

\begin{example}[$\lambda = \pm 1$]\label{ex:skew2}

\begin{equation*}
\sigma v_i = \lambda v_i, \ \ 1 \leq i \leq 2\ell, \qquad
\mathcal{N} v_i  =
  \begin{cases} 
   v_{i+1},    & 1 \leq i \leq \ell, \\
   -v_{i+1},& \ell+1 \leq i \leq 2\ell -1, \\
   0, & i = 2\ell.
  \end{cases}
\end{equation*}
Since $\lambda = \pm 1$, we have $\al_0 = 0$ or $-1/2$.
\end{example}

\begin{theorem}\label{thm:skew}
Consider a finite-dimensional vector space\/ $V$ with a nondegenerate skew-symmetric bilinear form\/ $(\cdot | \cdot)$. Let\/ $\sigma$ and\/ $\mathcal{N}$ be commuting linear operators satisfying \eqref{inv}, such that\/ $\sigma$ is invertible and semisimple and\/ $\mathcal{N}$ is nilpotent. Then\/ $V$ is an orthogonal direct sum of subspaces as in Examples \ref{ex:skew1} and \ref{ex:skew2}.
\end{theorem}
\begin{proof}
Denote by $V_\la$ the eigenspaces of $\si$. Since the form $(\cdot|\cdot)$ is nondegenerate and $\si$-invariant, it gives isomorphisms $V_\la \cong (V_{\la^{-1}})^*$, while $V_\la \perp V_\mu$ for $\la \neq \mu^{-1}$. Hence, we can assume that $V = V_\la \oplus V_{\la^{-1}}$ $(\la \neq \pm 1)$ or $V=V_{\pm 1}$.

In the first case, pick a basis $\{w_1,\ldots,w_d\}$ for $V_\la$ in which $\N$ is in lower Jordan form. Then $V_{\la^{-1}}$ has a basis $\{w_{d+1},\ldots,w_{2d}\}$ such that $(w_i|w_j) = \de_{i+j,2d+1}$ for all $i< j$, and $V$ becomes an orthogonal direct sum of subspaces as in Example \ref{ex:skew1}.

Now assume $V= V_{\pm 1}$. By a similar argument as above, we can assume that $V$ has the form $V'+V''$ where $V'$ has a basis $\{w_1,\ldots,w_d\}$, and $V''$ has a basis $\{w_{d+1},\ldots,w_{2d}\}$ such that
\begin{equation}\label{nchain}
\N \colon w_1 \mapsto w_2 \mapsto \cdots \mapsto w_d \mapsto 0,
\end{equation}
and $(w_i|w_{d+j}) =\delta_{i+j,d+1}$ for all $1 \leq i,j\leq d$. However, $V'$ and $V''$ may not be distinct, and we must consider two cases.

First, suppose $V'$ and $V''$ are distinct, and hence $V = V'\oplus V''$. When $d$ is odd, $\ph$ acts on $V'\oplus V''$ as in Example \ref{ex:skew1}. When $d$ is even, we make the following change of basis:
\begin{align*}
u'_i &= 
\begin{cases}
  \hfill \frac{1}{\sqrt{2}}(w_i + (-1)^{i+1}w_{d+i}) \hfill & \text{ if } 1 \leq i \leq \frac{d}{2},\\
  \hfill \frac{1}{\sqrt{2}}((-1)^{i+1}w_i + w_{d+i})  \hfill &\text{
    if } \frac{d}{2} +1 \leq i \leq d,\\
\end{cases}
\\
u''_{i}&=
\begin{cases}
  \hfill \frac{1}{\sqrt{2}}(w_i + (-1)^{i}w_{d+i}) \hfill & \text{ if } 1 \leq i \leq \frac{d}{2},\\
  \hfill \frac{1}{\sqrt{2}}((-1)^{i}w_i + w_{d+i})  \hfill &\text{
    if } \frac{d}{2} +1 \leq i \leq d.\\
\end{cases}
\end{align*}
Then $U'=\Span\{u'_1,\ldots,u'_d\}$ and $U''=\Span\{u''_1,\ldots,u''_d\}$ are as in Example \ref{ex:skew2}, and $V=U'\oplus U''$ is an orthogonal direct sum.

Finally, consider the case when $V=V'=V''$. Then $d=2\ell$ is even, and $V$ has a basis $\{w_1,\ldots,w_d\}$ such that \eqref{nchain} holds. Note that \eqref{inv} and \eqref{nchain} imply that $(w_i|w_j)=0$ whenever $i+j>2\ell+1$. With the appropriate rescaling we may assume $(w_1|w_{2\ell})=1$. Then 
\begin{align}\label{diagonal}
(w_i|w_{2\ell-i+1}) = (-1)^{i+1}, \qquad 1 \leq i \leq 2\ell.
\end{align}
A Gram--Schmidt process allows us to construct a new basis $w'_1,\ldots,w'_{2\ell}$ such that \eqref{nchain} and \eqref{diagonal} still hold and $(w'_i|w'_j) = 0$ when $i + j < 2\ell+1$. Thus $(w'_i|w'_j) = (-1)^{i+1}\de_{i+j,2\ell+1}$ for all $i,j$. Rescaling the basis vectors, we see that $V$ is as in Example \ref{ex:skew2}.
\end{proof}

Note that \prref{prop:symm} can be proved similarly to \thref{thm:skew}.

%%%%%%%%%%%%%%%%%%%%%%%%%%%%%%%%%%%%%%%%%%%%%%%%%
\section{Symplectic fermions}\label{ssfm}
%%%%%%%%%%%%%%%%%%%%%%%%%%%%%%%%%%%%%%%%%%%%%%%%%
In this section, we will continue to use the notation from \seref{sfsb}. In the case when $\lieh$ is even ($\lieh = \lieh_{\bar{0}}$), the free superbosons are known simply as free bosons and $B^1(\lieh)$ is called the Heisenberg vertex algebra. Its twisted logarithmic modules were described in \cite[Section 6]{Bak:2015}.
Now we will assume that $\lieh$ is odd, i.e., $\lieh = \lieh_{\bar{1}}$.
In this case, $B^1(\lieh)$ is called the \textit{symplectic fermion algebra} and denoted $SF$ (see \cite{Kau1,Kau2,Abe:2007}). 
Then the bilinear form $(\cdot|\cdot)$ on $\lieh$ is skew-symmetric.
For $\lieh$ and $\ph$ as in Examples \ref{ex:skew1} and \ref{ex:skew2}, we will describe explicitly the $\ph$-twisted affinization $\hat{\h}_\ph$ and its irreducible highest-weight modules $M_\ph(R)$, together with the action of the Virasoro algebra on them.

%%%%%%%%%%%%%%%%%%%%%%%%%%%%%%%%%%%%%%%%%%%%%%%%%
\subsection{Action of the  Virasoro algebra}\label{sactvir}

Choose a basis $\{v_i\}$ for $\h$ satisfying \eqref{skewbas}, where $\ph$ acts either as in Example \ref{ex:skew1} or \ref{ex:skew2}. Let $v^i = v_{2\ell-i+1}$ and $v^{\ell+i}=-v_{\ell-i+1}$ $(1 \leq i \leq \ell)$. The basis $\{v^i\}$ is dual to $\{v_i\}$ with respect to $(\cdot|\cdot)$, so that $(v_i|v^j)=\de_{i,j}$. Then
\begin{equation}\label{sbomega}
\om = \frac{1}{2}\sum_{i=1}^{2\ell}v^i_{(-1)}v_i=\sum_{i=1}^\ell v^i_{(-1)}v_i \in B^1(\lieh)
\end{equation}
is a conformal vector
with central charge $c=\sdim\lieh = \dim \lieh_{\bar{0}}-\dim \lieh_{\bar{1}}$. 
Since $\ph \om = \om$, the modes of $Y(\om,z)$ give a (untwisted) representation of the Virasoro Lie algebra on every $\ph$-twisted $B^1(\lieh)$-module 
(cf. \cite[Lemma 6.8]{Bak:2015}). 

The triangular decomposition \eqref{wtriangle} induces the following normal ordering on the modes of $\hhp$:
\begin{equation}\label{nopm}
\nopm{(at^m)(bt^n)} = 
\begin{cases}
(at^m)(bt^n) & \text{if }m \in \CC^-,\\
(-1)^{p(a)p(b)}(bt^n)(at^m) &\text{if } m \in \CC^+\cup\{0\}.
\end{cases}
\end{equation}
On the other hand, the normally ordered product $\nop{Y(a,z)Y(b,z)}$ of two logarithmic fields is defined
by placing the part of $Y(a,z)$ corresponding to powers $z^\ga$ with $\re\ga<0$ to the right of $Y(b,z)$
(see \cite[Section 3.3]{Bak:2015}). The two normal orderings of the modes are different in general, as we will see in the proof of the next proposition.

\begin{proposition}\label{prop:bvirfield}
Assume\/ $\lieh$ and\/ $\ph$ are as in Example \ref{ex:skew1} or \ref{ex:skew2}. Then in every\/ $\ph$-twisted module of\/ $SF = B^1(\lieh)$, we have
\begin{equation}\label{sfLk}
L_k = \sum_{i=1}^\ell \sum_{m\in\al_0+\ZZ} \nopm{(v^it^{-m})(v_it^{k+m})}+\de_{k,0}\frac{\ell}{2}\al_0(\al_0+1)I.
\end{equation}
\end{proposition}
\begin{proof}
Assume $v_1,\ldots,v_\ell \in \lieh_\al$ and $v_{\ell+1},\ldots,v_{2\ell} \in \lieh_\beta$, where $\be=-\al$. Using \cite[Lemma 5.8]{Bak:2015}, (\ref{nthprodb}), the skew-symmetry of $(\cdot|\cdot)$, and the fact that $(\mathcal{N}v^i|v_i)=0$ for $1 \leq i \leq \ell$, we obtain
\begin{equation}\label{genomfield}
Y(\om,z) = \sum_{i=1}^\ell\nop{X(v^i,z)X(v_i,z)}+z^{-2}\ell \binom{\be_0}{2}I.
\end{equation}
If $\re \al_0=0$, then $\be_0=-\al_0 \in \CC^+\cup\{0\}$. So the normal ordering in \eqref{genomfield} coincides with $\eqref{nopm}$, and $\binom{\be_0}{2} =\frac{1}{2}\al_0(\al_0+1).$
Thus $L_k$ is given by \eqref{sfLk}.

Now assume $\re \al_0 <0$. Then $\be_0 = -\al_0-1 \in \CC^-$, and the ordering of the modes in \eqref{genomfield} differs from \eqref{nopm} when $k=0$ for
\begin{equation*}
-(v_it^{-\be_0})(v^it^{\be_0})=\nopm{(v^it^{\be_0})(v_it^{-\be_0})}+\be_0I.
\end{equation*}
Finally, we note that $\be_0+\binom{\be_0}{2} = \frac{1}{2}\al_0(\al_0+1)$. Thus after reordering to match (\ref{nopm}), $L_k$ is given by \eqref{sfLk}.
\end{proof}

\begin{remark}\label{rem:bvir}
Similarly, when $\h$ is even and $\ph$ is as in Example \ref{ex:symm1}, we have
\begin{equation}\label{bLk1}
L_k = \sum_{i=1}^\ell \sum_{m\in\al_0+\ZZ} \nopm{(v^it^{-m})(v_it^{k+m})}-\de_{k,0}\frac{\ell}{2}\al_0(\al_0+1)I
\end{equation}
in any $\ph$-twisted $B^1(\lieh)$-module. In the case of Example \ref{ex:symm2}, we have
\begin{equation}\label{bLk2}
L_k = \frac12\sum_{i=1}^d \sum_{m\in\al_0+\ZZ} \nopm{(v^it^{-m})(v_it^{k+m})}-\de_{k,0}\frac{d}{4}\al_0(\al_0+1)I.
\end{equation}
These results agree with \cite[Section 6]{Bak:2015}.
\end{remark}

Note that the normal orderings in \eqref{sfLk}, \eqref{bLk1}, \eqref{bLk2} can be omitted for $k\ne0$. In the following subsections, we will compute explicitly the actions of $\hhp$ and $L_0$ on $M_\ph(R)$ when $\h$ is odd as in Examples \ref{ex:skew1}, \ref{ex:skew2}.

%%%%%%%%%%%%%%%%%%%%%%%%%%%%%%%%%%%%%%%%%%%%%%%%%
\subsection{The case of {E}xample \ref{ex:skew1}} \label{sfcase1}

Recall that for any $\ph$-twisted $SF$-module, the logarithmic fields are given by \eqref{tfield}. 
Assume that $\N$ acts on $\lieh$ as in Example \ref{ex:skew1}.
Since this action is the same as in Example \ref{ex:symm1}, the logarithmic fields $Y(v_j,z)$ are the same as in \cite [Section 6.4]{Bak:2015}:
\begin{equation}\label{tfields1}
\begin{split}
Y(v_j,z)& = \sum_{i=j}^\ell\sum_{m\in\al_0+\ZZ} \frac{(-1)^{i-j}}{(i-j)!} \ze^{i-j} (v_it^m) z^{-m-1},\\
Y(v_{\ell+j},z)& = \sum_{i=j}^\ell\sum_{m\in-\al_0+\ZZ} \frac1{(i-j)!} \ze^{i-j} (v_{\ell+i}t^m) z^{-m-1}, 
\end{split}
\end{equation}
for $1\leq j\leq\ell$.

The Lie superalgebra $\hat{\h}_\ph$ is spanned
by an even central element $K$ and odd elements $v_it^{m+\al_0}$, $v_{\ell+i}t^{m-\al_0}$ ($1 \leq i \leq \ell$, $m\in \ZZ$), where $\al_0 \in \CC^-\cup \{0\}$ and $-1<\re \al_0\leq 0$.  By  \eqref{waffbrack}, the only nonzero brackets in $\hhp$ are given by:
\begin{equation}\label{sfbrack1}
%\begin{split}
[v_it^{m+\al_0},v_jt^{n-\al_0}]=(m+\al_0)\delta_{m,-n}\delta_{i+j,2\ell+1}K+ \delta_{m,-n} %(1-\delta_{i,\ell})
\delta_{i+j,2\ell}K,
%[v_{\ell+i}t^{m-\al_0},v_jt^{n+\al_0}] &=-(m-\al_0)\delta_{m,-n}\delta_{i +j,\ell+1}K +\delta_{m,-n}\delta_{i+j,\ell}K, 
%\end{split}
\end{equation}
for $1 \leq i \leq \ell$, $\ell+1 \leq j \leq 2 \ell$, $m,n \in \ZZ$. 
Notice that the elements of $\hhp^-$ act as creation operators on $M_\ph(R)$.
Throughout the rest of the section, we will represent them as anti-commuting variables as follows:
\begin{equation}\label{xi}
v_it^{-m+\al_0} = \xi_{i,m},  \qquad v_jt^{-n-\al_0} = \xi_{j,n}, 
\end{equation}
for $1\leq i\leq\ell$, $\ell+1\leq j \leq 2\ell$, and $m \geq 0$, $n\geq 1$. 
% We note that when $\al_0=0$, technically $m\geq 1$ in \eqref{xi}, since $v_it^0 \notin \hhp^-$. Though we will see below that \eqref{xi} still holds for $m=0$.

The precise triangular decomposition \eqref{wtriangle} depends on whether $\alpha_0 \in \CC^-$ or $\alpha_0=0$.
Suppose first that $\al_0 \in \CC^-$. Then $\hhp^0=\CC K$ and $R = \CC$. Equations \eqref{sfbrack1} and \eqref{xi} imply that 
\begin{equation}\label{sfverma1}
M_\ph(R) \cong \bigwedge(\xi_{i,m},\xi_{\ell+i,m+1})_{1\leq i \leq \ell,\, m=0,1,2,\ldots}.
\end{equation}
Using the commutation relations \eqref{sfbrack1} and the fact that $\hhp^+R = 0$, we obtain the action of $\hhp^+$ on $M_\ph(R)$:
\begin{equation*}%\label{action1}
\begin{split}
v_it^{m+\al_0}&=(m+\al_0)\partial_{\xi_{2\ell-i+1,m}}+(1-\delta_{i,\ell})\partial_{\xi_{2\ell-i,m}},\\
v_{\ell+i}t^{n-\al_0}&=-(n-\al_0)\partial_{\xi_{\ell-i+1,n}}+(1-\delta_{i,\ell})\partial_{\xi_{\ell-i,n}}, 
\end{split}
\end{equation*}
where $1 \leq i \leq \ell$, $m\geq 1$, $n \geq 0$. By Proposition \ref{prop:bvirfield}, the action of $L_0$ is
\begin{equation}\label{SFL1}
\begin{split}
L_0 =&\sum_{i=1}^\ell\sum_{m =0}^\infty \xi_{i,m}\Big((m-\al_0)\partial_{\xi_{i,m}}-(1-\delta_{i,1})\partial_{\xi_{i-1,m}}\Big) \\
&\quad +\sum_{i=1}^\ell\sum_{m =1}^\infty  \xi_{\ell+i,m}\Big((m+\al_0)\partial_{\xi_{\ell+i,m}}+(1-\delta_{i,1})\partial_{\xi_{\ell+i-1,m}}\Big) \\
&\quad +\frac{\ell}{2}\al_0(\al_0+1)I. 
\end{split}
\end{equation}

Now we consider the case when $\al_0=0$. Then $\hhp^0  = \Span\{v_it^0\}_{1 \leq i \leq 2\ell}\oplus\CC K$. 
We let 
\begin{equation*}
R= \bigwedge(\xi_{i,0},\xi_{2\ell,0})_{1\leq i \leq \ell},
\end{equation*}
where the action of $\hhp^0$ on $R$ is given by
\begin{equation}\label{ract1}
\begin{split}
v_it^0 &= \xi_{i,0}, \qquad\quad 1 \leq i \leq \ell \;\text{ or }\; i=2\ell, \\
v_{j}t^0 &= \partial_{\xi_{2\ell-j,0}}, \qquad \ell+1 \leq j \leq 2\ell-1.
\end{split}
\end{equation}
Therefore, by \eqref{xi},
\begin{equation}\label{sfverma2}
M_\ph(R) \cong \bigwedge(\xi_{i,m},\xi_{2\ell,m},\xi_{j,m+1})_{1\leq i \leq \ell,\,\ell+1\leq j \leq 2\ell-1, \,m=0,1,2,\ldots},
\end{equation}
where the action of $\hhp^+$ is given by 
\begin{align*}
v_it^m & = m \partial_{\xi_{2\ell-i+1,m}}+(1-\delta_{i,\ell})\partial_{\xi_{2\ell-i,m}},\\
v_{\ell + i}t^m & =-m\partial_{\xi_{\ell-i+1,m}}+(1-\delta_{i,\ell})\partial_{\xi_{\ell-i,m}},
\end{align*}
for $1 \leq i \leq \ell$ and $m\geq 1$.  The action of $L_0$ is
\begin{align*}
L_0 = \sum_{i=1}^\ell & \sum_{m=0}^\infty \xi_{i,m}\Big(m\partial_{\xi_{i,m}}-(1-\delta_{i,1})\partial_{\xi_{i-1,m}}\Big)\\ &+\sum_{i=1}^\ell\sum_{m=1}^\infty
\xi_{\ell+i,m}\Big(m\partial_{\xi_{\ell+i,m}}+(1-\delta_{i,1})\partial_{\xi_{\ell+i-1,m}}\Big) 
 -\xi_{1,0}\xi_{2\ell,0}.
\end{align*}

%%%%%%%%%%%%%%%%%%%%%%%%%%%%%%%%%%%%%%%%%%%%%%%%%
 \subsection{The case of {E}xample \ref{ex:skew2}}

Let $\lieh$ be as in Example \ref{ex:skew2}. 
Then, by \eqref{tfield}, in any $\ph$-twisted $SF$-module,
\begin{equation}\label{tfields2}
\begin{split}
Y(v_j,z)  = &\sum_{i=j}^\ell\sum_{m \in\al_0 +\ZZ} \frac{(-1)^{i-j}}{(i-j)!} \ze^{i-j} (v_it^m)z^{-m-1} \\
&+  (-1)^{j-\ell+1}\sum_{i=\ell+1}^{2\ell}\sum_{m \in \al_0+\ZZ} \frac1{(i-j)!} \ze^{i-j} (v_it^m)z^{-m-1}, \\
Y(v_{\ell+j},z) = & \sum_{i=j}^\ell\sum_{m\in\al_0 +\ZZ}\frac1{(i-j)!} \ze^{i-j}  (v_{\ell+i} t^m) z^{-m-1},
\end{split} 
\end{equation}
for $1 \leq j \leq \ell$ and $\al_0 = 0$ or $-1/2$.

The Lie superalgebra $\hhp$ is spanned by an even central element $K$ and odd elements $v_it^m$ $(1 \leq i \leq 2\ell, m \in \al_0 + \ZZ)$. The brackets in $\hhp$ are:
\begin{equation*}%\label{sfbrack2}
\begin{split}
[v_it^m,v_jt^n] & = m\de_{m,-n}\de_{i+j,2\ell+1}K +\de_{m,-n}(1-2\de_{i,\ell})\de_{i+j,2\ell}K,\\
[v_{\ell+i}t^m,v_jt^n]&= -m\de_{m,-n}\de_{i+j,\ell+1}K +\de_{m,-n}\de_{i+j,\ell}K, 
\end{split}
\end{equation*}
for $1 \leq i \leq \ell$, $1 \leq j \leq 2\ell$, $m,n \in \al_0+\ZZ$. 
To determine explicitly $\hhp^0$, we need to consider separately the cases
$\al_0=0$ or $-1/2$.

First, we assume $\al_0=0$. Then $\hhp^0  = \Span\{v_it^0\}_{1 \leq i \leq 2\ell}\oplus\CC K$. We let
\begin{align*}
R= \bigwedge(\xi_{i,0},\xi_{2\ell,0})_{1\leq i \leq \ell} \qquad \Bigl(\text{where } \xi_{\ell,0}^2 = -\frac12\Bigr),
\end{align*}
with the action of $\hhp^0$ given by  \eqref{ract1}. Again, we will let the creation operators from $\hhp^-$ act by \eqref{xi}.
Thus $M_\ph(R)$ is again as in  \eqref{sfverma2} but with $\xi_{\ell,0}^2 = -1/2$.
% \begin{equation*}
% M_\ph(R) \cong \bigwedge(\xi_{i,m},\xi_{2\ell,m},\xi_{j,m+1})_{1\leq i \leq \ell,\ell+1\leq j \leq 2\ell-1, m= 0,1,2,\ldots}, \qquad (\text{where } \xi_{\ell,0}^2 = -1/2).
% \end{equation*}
The action of $\hhp^+$ on $M_\ph(R)$ is given by 
\begin{align*}
v_it^m & = m \partial_{\xi_{2\ell-i+1,m}}+(1-2\delta_{i,\ell})\partial_{\xi_{2\ell-i,m}}, \\
v_{\ell + i}t^m & = -m \partial_{\xi_{\ell-i+1,m}}+(1-\delta_{i,\ell})\partial_{\xi_{\ell-i,m}},
\end{align*}
for $1 \leq i \leq \ell$, $m\geq1$. The action of $L_0$ is
\begin{align*}
L_0 = \sum_{i=1}^\ell & \sum_{m=0}^\infty \xi_{i,m}\Big(m\partial_{\xi_{i,m}}-(1-\delta_{i,1})\partial_{\xi_{i-1,m}}\Big)\\ &+\sum_{i=1}^\ell\sum_{m=1}^\infty
\xi_{\ell+i,m}\Big(m\partial_{\xi_{\ell+i,m}}+(1-2\delta_{i,1})\partial_{\xi_{\ell+i-1,m}}\Big) 
 -\xi_{1,0}\xi_{2\ell,0}.
\end{align*}

Second, we consider the case when $\al_0 = -1/2$. %in \exref{ex:skew2}.
Then $\hhp^0 = \CC K$. We represent the elements of $\hhp^-$ on $M_\ph(R)$ as  $v_it^{-m-1/2}=\xi_{i,m}$ for $1 \leq i \leq 2\ell$ and $m \geq 0$. Then
\begin{equation*}
M_\ph(R) \cong \bigwedge(\xi_{i,m})_{1\leq i \leq 2\ell,\, m=0,1,2,\ldots}.
\end{equation*}
The action of $\hhp^+$ on $M_\ph(R)$ is
\begin{align*}
v_it^{m+1/2} & = \Big(m+\frac{1}{2}\Big) \partial_{\xi_{2\ell-i+1,m}}+(1-2\de_{i,\ell})\partial_{\xi_{2\ell-i,m}},\\
v_{\ell+i}t^{m+1/2} & = -\Big(m+\frac{1}{2}\Big) \partial_{\xi_{\ell-i+1,m}}+(1-\de_{i,\ell})\partial_{\xi_{\ell-i,m}},\\
\end{align*}
for $1 \leq i \leq \ell$, $m\geq 0$. The action of $L_0$ is
\begin{align*}
L_0 = \sum_{i=1}^\ell & \sum_{m=0}^\infty \xi_{i,m}\Big(\Big(m+\frac{1}{2}\Big)\partial_{\xi_{i,m}}-(1-\delta_{i,1})\partial_{\xi_{i-1,m}}\Big)\\ &+\sum_{i=1}^\ell\sum_{m=0}^\infty
\xi_{\ell+i,m}\Big(\Big(m+\frac{1}{2}\Big)\partial_{\xi_{\ell+i,m}}+(1-2\delta_{i,1})\partial_{\xi_{\ell+i-1,m}}\Big) 
 -\frac{\ell}{8}I.
\end{align*}

\begin{remark}\label{triplet}
Let $\dim \lieh=2$ in \exref{ex:skew2}.
The \textit{triplet algebra} $\mathfrak{W}(1,2) \subset SF$ is generated by the elements (see \cite{Kau3,CR}):
\begin{equation*}%\label{tripletvecs}
W^+ = -v_{1(-2)}v_1, \quad W^0 = -v_{1(-2)}v_2 - v_{2(-2)}v_1, \quad W^- =-v_{2(-2)}v_2.
\end{equation*}
Then $\si=I$ on $\mathfrak{W}(1,2)$ and 
$\N \colon W^+ \mapsto W^0 \mapsto 2W^- \mapsto 0$.
Hence, the restriction of any $\ph$-twisted module of $SF$ to $\mathfrak{W}(1,2)$ is a $\ph$-twisted module of $\mathfrak{W}(1,2)$, in which
$Y(W^-,z)$ is independent of $\ze$ while the fields $Y(W^+,z)$ and $Y(W^0,z)$ are logarithmic. %Explicitly,
%\begin{align*}
%Y(W^-,z) &= X(W^-,z), \\
%Y(W^0,z) &= X(W^0,z) - 2\ze X(W^-,z), \\
%Y(W^+,z) &= X(W^+,z) - \ze X(W^0,z) + \ze^2 X(W^-,z).
%\end{align*}
\end{remark}

%%%%%%%%%%%%%%%%%%%%%%%%%%%%%%%%%%%%%%%%%%%%%%%%%
\section{Free superfermions}\label{ssferm}
%%%%%%%%%%%%%%%%%%%%%%%%%%%%%%%%%%%%%%%%%%%%%%%%%
In this section, we study twisted logarithmic modules of the free superfermion algebras.
First, let us review the definition of free superfermions given in \cite{K2}. Let $\liea$ be an abelian Lie superalgebra with $\dim \liea = d < \infty$, and $(\cdot|\cdot)$ be a nondegenerate even anti-supersymmetric bilinear form on $\liea$. Thus $(b|a)=-(-1)^{p(a)p(b)}(a|b)$ and $(\h_{\bar0}|\h_{\bar1})=0$. The \emph{Clifford affinization} of $\liea$ is the Lie superalgebra
\begin{equation*}
C_{\liea} = \liea[t,t^{-1}]\oplus \CC K
\end{equation*}
with commutation relations
\begin{equation}\label{sfbrackets}
[a t^m,b t^n] = (a|b)\delta_{m,-n-1}K,  \qquad [C_{\liea},K]=0
\end{equation}
for $m,n \in \ZZ$, where $p(at^m)=p(a)$ and $p(K)=\bar0$.
The \emph{free superfermions}
\begin{equation*}
a(z) = \sum_{m\in \ZZ}a_{(m)}z^{-m-1}, \qquad a_{(m)}=at^m,
\end{equation*}
have OPEs given by
\begin{equation*}%\label{fOPE}
a(z)b(w) \sim \frac{(a|b)K}{z-w}.
\end{equation*}

The (generalized) Verma module 
\begin{equation*}
V = \ind_{\liea[t]\oplus \CC K}^{C_{\liea}} \CC
\end{equation*}
is constructed by letting $\liea[t]$ act trivially on $\CC$ and $K$ act as 1. Then $V$ has the structure of a vertex algebra called the 
\textit{free superfermion algebra} and denoted $F^1(\liea)$.
The brackets \eqref{sfbrackets} are equivalent to the following $n$-th products in $F^1(\liea)$:
\begin{equation}\label{fnprod}
a_{(0)}b = (a|b)\vac, \qquad a_{(j)}b = 0 \quad (j \geq 1),
\end{equation}
where $\vac$ is the vacuum vector.
In the even case ($\liea=\liea_{\bar0}$), the free superfermions are also known as symplectic bosons or as the bosonic ghost system (\emph{$\be\ga$-system}). 
In the odd case ($\liea=\liea_{\bar1}$), they are just called \emph{free fermions}.

%%%%%%%%%%%%%%%%%%%%%%%%%%%%%%%%%%%%%%%%%%%%%%%%%
\subsection{Twisted logarithmic modules of $F^1(\liea)$} 

Letting $\ph$ be an automorphism of $\liea$ such that $(\cdot|\cdot)$ is $\ph$-invariant, we write as before $\ph = \sigma e^{-2\pi\ii\N}$, and denote the eigenspaces of $\sigma$ by 
\begin{equation*}
\liea_\al = \{a \in \liea\, | \,\sigma a = e^{-2 \pi\ii \al}a \}, \qquad \al \in \CC/\ZZ.
\end{equation*}

\begin{definition}\label{def:tcaff}
The \emph{$\ph$-twisted Clifford affinization} $\caph$ is the Lie superalgebra spanned by elements $a t^m$ $(a \in \liea_\al, m \in \al)$ with $p(at^m)=p(a)$ and an even central element $K$.
The Lie superbracket in $\caph$ is given by
\begin{equation}\label{caffbrack}
[a t^m,b t^n] = \delta_{m,-n-1}(a|b)K, \qquad [K,a t^m] = 0,
\end{equation}
for $a \in \liea_\al$, $b \in \liea_\be$, $m\in \al$, $n\in \be$.
\end{definition} 

\begin{remark}\label{rem:tcaff}
Since the brackets \eqref{caffbrack} do not depend on $\N$, we have $\caph = (C_{\liea})_\si$.
In particular, $\caph = C_{\liea}$ if $\ph = e^{-2\pi\ii\N}$.
\end{remark}

As in the case of superbosons, $\ph$ naturally induces automorphisms of $C_{\liea}$ and $F^1(\liea)$. 
As before, a $\caph$-module $W$ will be called \textit{restricted} if for every $a \in \liea_\al$, $m \in \al$, $v \in W$, there is an integer $L$ such that $(at^{m+i})v =0$ for all $i \in \ZZ,$ $i \geq L$. 

\begin{theorem}\label{thm:tcaff}
Every\/ $\ph$-twisted\/ $F^1(\liea)$-module\/ $W$ has the structure of a restricted\/ $\caph$-module with\/ 
$(a t^m)v=a_{(m+\N)}v$ for\/ $a \in \liea_\al$, $m \in \al$, $v \in W$.
Conversely, every restricted\/ $\caph$-module uniquely extends to a\/ $\ph$-twisted\/ $F^1(\liea)$-module.
\end{theorem}

The proof of the theorem is identical to that of \cite[Theorem 6.3]{Bak:2015} and is omitted.
It follows from $\caph = (C_{\liea})_\si$ that every
$\ph$-twisted $F^1(\liea)$-module $W$ has the structure of a $\si$-twisted $F^1(\liea)$-module,
and vice versa. More precisely, if $Y\colon F^1(\liea)\to\LF(W)$ is the state-field correspondence as a $\ph$-twisted module, then the state-field correspondence as a $\si$-twisted module is given by the map $X$ from \eqref{tfieldx} for $a\in\liea$. Conversely, given $X$, we can determine $Y$ from \eqref{tfield} for $a\in\liea$. 
However, the relationship is more complicated for elements $a\in F^1(\liea)$ that are not in the generating set $\liea$. In particular, we will see below that the action of the Virasoro algebra is different, so that $L_0$ is not semisimple in a $\ph$-twisted module while it is semisimple in a $\si$-twisted module.

We will split $\CC$ as a disjoint union of subsets $\CC_{-\frac12}^+$, $\CC_{-\frac12}^-$ and $\{-\frac12\}$ where 
\begin{equation}\label{cplus2}
\CC_{-\frac12}^+ = -\frac12 + \CC^+  , \qquad \CC_{-\frac12}^- = -\frac12 - \CC^+ ,
\end{equation}
and $\CC^+$ is given by \eqref{cplus}.
The $\ph$-twisted Clifford affinization $\caph$ has a triangular decomposition 
\begin{equation}\label{catr}
\caph = \caph^- \oplus \caph^0 \oplus \caph^+,
\end{equation}
 where
\begin{equation*}
(C_\liea)^\pm_\ph = \Span\bigl\{ at^m \,\big|\, a \in \liea_\al, \, \al \in \CC/\ZZ, \, m \in \al \cap \CC_{-\frac12}^\pm \bigr\}
\end{equation*}
and
\begin{equation*}
(C_\liea)_\ph^0 = \Span\bigl\{at^{-\frac12} \,\big|\, a \in \liea_{-\frac12} \bigr\}\oplus \CC K.
\end{equation*}
Starting from a $(C_\liea)_\ph^0$-module $R$ with $K=I$, the (generalized) \emph{Verma module} is defined by
\begin{equation*}%\label{bverma}
M_\ph(R) = \ind_{(C_\liea)_\ph^+ \oplus (C_\liea)_\ph^0}^{\caph}R ,
\end{equation*}
where $(C_\liea)_\ph^+$ acts trivially on $R$. These are $\ph$-twisted $F^1(\liea)$-modules, and in the following sections we will realize them explicitly as Fock spaces and will determine the action of the Virasoro algebra on them.

%%%%%%%%%%%%%%%%%%%%%%%%%%%%%%%%%%%%%%%%%%%%%%%%%
\subsection{Action of the Virasoro algebra}\label{subsec:ovir}

Pick bases $\{v_i\}$ and $\{v^i\}$ of $\liea$ such that $p(v_i)=p(v^i)$ and $(v_i|v^j) = \delta_{i,j}$. Then 
\begin{equation}\label{fomega}
\omega = \frac{1}{2}\sum_{i=1}^d v^i_{(-2)} v_i \in F^1(\liea),
\qquad\quad d=\dim\liea,
\end{equation}
is a conformal vector with central charge $c = -\frac{1}{2}\sdim \liea$. 

Let $\mathcal{S}\colon\liea\to \liea$ be the linear operator given by $\mathcal{S} a = \al_0 a$ for $a \in \liea_\al$,
where, as before, $\al_0\in\al$ is such that $-1<\re \al_0 \leq 0$.
In the next result, we use the normally ordered product from \cite[Section 3.3]{Bak:2015}
(cf.\ \seref{sactvir}).

\begin{proposition}\label{prop:fvir}
In every\/ $\ph$-twisted\/ $F^1(\liea)$-module, we have
\begin{align*}
2 Y(\om,z) = \sum_{i=1}^d \nop{&\bigl(\partial_z X(v^i,z)\bigr) X(v_i,z)} 
- z^{-1} \sum_{i=1}^d \nop{X(\N v^i,z) X(v_i,z)} \\
&- z^{-2} \str \binom{\mathcal{S}}{2} I,
\end{align*}
where $\str$ denotes the supertrace.
\end{proposition}
\begin{proof}
Recall that in any vertex algebra, $(Ta)_{(j)}b=-ja_{(j-1)}b$, where $T$ is the translation operator (see e.g.\ \cite{K2}).
By replacing $a$ with $Ta$ in \cite[Lemma 5.8]{Bak:2015} and using \cite[(4.3)]{Bak:2015}, we obtain
\begin{align*}
\nop{\bigl(D_z Y(a,z)\bigr) Y(b,z)} %&= \nop{Y(Ta,z) Y(b,z)} \\
= -\sum_{j=-1}^{N-1} j z^{-j-1} Y\Bigl( \bigl( \binom{\mathcal{S}+\N}{j+1} a\bigr)_{(j-1)} b,z \Bigr)
\end{align*}
for sufficiently large $N$ (depending on $a,b$). Due to \eqref{fnprod}, when $a,b\in\liea$, the right-hand side reduces to
\begin{equation*}
Y(a_{(-2)} b,z) - z^{-2} \Bigl( \binom{\mathcal{S}+\N}{2} a \Big| b \Bigr) I.
\end{equation*}
Now using \eqref{dzdze} and \eqref{tfield}, we observe that
\begin{equation}\label{diffy}
D_z Y(a,z)\big|_{\ze=0} = \partial_z X(a,z) - z^{-1} X(\N a,z).
\end{equation}
Finally, we note that
\begin{equation*}
\sum^d_{i=1}\bigl(\binom{\mathcal{S}}{2} v^i \big| v_i\bigr) = -\sum_{i=1}^d(-1)^{p(v^i)}\bigl(v_i\big|\binom{\mathcal{S}}{2} v^i\bigr) = -\str \binom{\mathcal{S}}{2}.
\end{equation*}
Then the rest of the proof is as in \cite[Lemma 6.4]{Bak:2015}.
\end{proof}

%%%%%%%%%%%%%%%%%%%%%%%%%%%%%%%%%%%%%%%%%%%%%%%%%
%\subsection{Other Virasoro fields}

Now let us assume that $\liea$ can be written as the direct sum of two isotropic subspaces $\liea^-=\Span \{v_i\}$ and $\liea^+=\Span\{v^i\}$ $(1 \leq i \leq \ell)$, where, as before, 
$(v_i|v^j)=\de_{i,j}$ and $d=\dim\liea=2\ell$.
Following \cite[Section 3.6]{K2}, we let
\begin{equation}\label{omla}
\om^\la = (1-\la)\om^+ + \la\om^- \qquad (\la \in \CC),
\end{equation}
where
\begin{equation*}%\label{om01}
\om^+ = \sum_{i=1}^\ell v^i_{(-2)}v_i, \qquad 
\om^- = -\sum_{i=1}^\ell (-1)^{p(v_i)} {v_{i}}_{(-2)}v^i.
\end{equation*}
Then $\om^\la$ is a conformal vector in $F^1(\liea)$  with central charge 
\begin{equation*}
c_\la = (6\la^2-6\la+1) \sdim\liea.
\end{equation*}
In particular, $\om^{1/2}$ coincides with \eqref{fomega}.
We denote the corresponding family of Virasoro fields as
\begin{equation*}%\label{virlam}
L^\la(z) = Y(\om^\la,z) = (1-\la)L^+(z)+\la L^-(z).
\end{equation*}
Their action can be derived from the proof of \prref{prop:fvir} as follows.

\begin{corollary}\label{fvirother}
If\/ $\ph(\om^+)=\om^+$,
then in every\/ $\ph$-twisted\/ $F^1(\liea)$-module
\begin{align*}
L^+(z) = \sum_{i=1}^\ell \nop{&\bigl(\partial_z X(v^i,z)\bigr) X(v_i,z)} 
- z^{-1} \sum_{i=1}^\ell \nop{X(\N v^i,z) X(v_i,z)} \\
&- z^{-2} \str \binom{\mathcal{S^+}}{2} I,
\end{align*}
where\/ $\mathcal{S^+}$ is the restriction of\/ $\mathcal{S}$ to\/ $\liea^+$.
\end{corollary}

If the automorphism $\ph$ is as in \exref{ex:skew2}, then a short calculation gives
$\N(\om^\la) = (2\la-1) {v_{\ell+1}}_{(-2)}v_{\ell+1}$. This implies that only the modes of $Y(\omega^{1/2},z)=L^{1/2}(z)$ yield an untwisted representation of the Virasoro algebra on a $\ph$-twisted $F^1(\liea)$-module. 
If $\ph$ is as in Examples \ref{ex:symm1} or \ref{ex:skew1}, then $\N(\om^\la)=0$ and $L^\la(z)$ yields an untwisted representation of the Virasoro algebra for any $\la \in \CC$. 

\begin{remark}
In the special case when $\liea$ is even with $\dim\liea=2$ (i.e., when we have a $\be\ga$-system of rank $1$), the above Virasoro fields resemble but are different from
those of \cite[(3.15)]{An}.
\end{remark}

%%%%%%%%%%%%%%%%%%%%%%%%%%%%%%%%%%%%%%%%%%%%%%%%%
\subsection{Subalgebras of free superfermions}

Suppose that $\dim\liea=2\ell$ as in Examples \ref{ex:symm1}, \ref{ex:skew1}.
It is well-known that the elements 
\begin{equation*}
u_i=v_i \in F^1(\liea), \quad u_{\ell+i}=T v_{\ell+i} \in F^1(\liea) \qquad (1 \leq i \leq \ell)
\end{equation*}
(where $T$ is the translation operator) are generators of the \emph{free superboson algebra} $B^1(\liea)\subset F^1(\liea)$.
This is
the Heisenberg vertex algebra in the case of \exref{ex:symm1}, and the symplectic fermion algebra $SF$ in the case of \exref{ex:skew1} (see e.g.\ \cite{Kau1}). 
Here $u_i$ plays the role of $v_i$ from \seref{ssfm} and from \cite[Section 6.3]{Bak:2015}.

When $\ph$ is the automorphism of $F^1(\liea)$ from Examples \ref{ex:symm1}, \ref{ex:skew1}, then $\ph$ restricts to an automorphism of $B^1(\liea)$
of the same type, since $\ph$ commutes with $T$. Thus, any $\ph$-twisted $F^1(\liea)$-module restricts to a $\ph$-twisted $B^1(\liea)$-module.
In such a module, the logarithmic fields corresponding to the generators are given by (cf.\ \cite[(4.3)]{Bak:2015}):
\begin{equation*}
Y(u_i,z) = Y(v_i,z), \quad Y(u_{\ell+i},z) = D_zY(u_{\ell+i},z) \qquad (1 \leq i \leq \ell). 
\end{equation*}
Then ${u_{i}}_{(m+\N)}={v_{i}}_{(m+\N)}$ for $m \in \al_0+\ZZ$, and using \eqref{diffy} we obtain
\begin{equation*}%\label{subalmodes}
{u_{\ell+i}}_{(m+\N)}=-m{v_{\ell+i}}_{(m-1+\N)}+(1-\de_{i,\ell}){v_{\ell+i+1}}_{(m-1+\N)}
\end{equation*}
for $m \in -\al_0+\ZZ$.
The action of these modes on $M_\ph(R)$ is related to the $\ph$-twisted modules constructed in \seref{ssfm} and \cite[Section 6.3]{Bak:2015}
by a linear change of variables.

The free superboson algebra $B^1(\liea)$ has a conformal vector (cf.\ \eqref{sbomega})
\begin{equation*}
\om'= \sum_{i=1}^\ell {u_{2\ell-i+1}}_{(-1)}u_i.
\end{equation*}
Since
\begin{equation*}
{u_{2\ell+1-i}}_{(-1)}u_i=(Tv_{2\ell+1-i})_{(-1)}v_i={v_{2\ell+1-i}}_{(-2)}v_i = v^i_{(-2)}v_i,
\end{equation*}
we have $\om'=\om^+ \in F^1(\liea)$ (see \eqref{omla}).
It follows that the action of $Y(\om',z)$ on $M_\ph(R)$ is equivalent to the action of $L^+(z)$.
The actions of $L(z) = L^{1/2}(z)$ and $L^+(z)$ will be computed explicitly in the following two sections.

Another important subalgebra of $F^1(\liea)$ is the \emph{$\W_{1+\infty}$-algebra} \cite{KR}. It is generated by the following elements similar to $\om^+$:
\begin{equation}\label{nun1}
\nu^n = \sum_{i=1}^\ell v^i_{(-n)}v_i, \qquad n=1,2,3,\dots
\end{equation}
so that $\nu^2=\om^+$ and $\nu^1$ generates the Heisenberg algebra.
The automorphism $\ph$ of $F^1(\liea)$ from Examples \ref{ex:symm1}, \ref{ex:skew1} satisfies $\ph(\nu^n)=\nu^n$ for all $n$.
Therefore, any $\ph$-twisted $F^1(\liea)$-module restricts to an (untwisted) module of $\W_{1+\infty}$. The fields $Y(\nu^n,z)$ in such a module can be computed
as in \prref{prop:fvir} and \coref{fvirother}:
\begin{equation}\label{nun2}
\begin{split}
Y(\nu^n,z) = \frac1{(n-1)!} & \sum_{i=1}^\ell \, \nop{X\bigl( \bigl( \partial_z - z^{-1} \N \bigr)^{n-1} v^i,z\bigr) X(v_i,z)} \\
& +(-1)^{n-1} z^{-n} \str \binom{\mathcal{S^+}}{n} I.
\end{split}
\end{equation}

Other important realizations by free superfermions are those of \emph{classical affine Lie (super)algebras} \cite{F,KP,FF4,KW1}.
Here we discuss just one example.
Let us assume, as before, that $\liea$ can be written as the direct sum of two isotropic subspaces $\liea^-=\Span \{v_i\}$ and $\liea^+=\Span\{v^i\}$ $(1 \leq i \leq \ell)$,
where $(v_i|v^j)=\de_{i,j}$.
We label the basis vectors so that the odd part $(\liea^-)_{\bar1}$ is spanned by $\{v_i\}_{i=1,\dots,m}$, and the even part $(\liea^-)_{\bar0}$ is spanned by $\{v_i\}_{i=m+1,\dots,m+n}$
where $\ell=m+n$. Then, by \cite[Proposition 3.1]{KW1}, the elements $e_{ij} = {v_i}_{(-1)}v^j$ ($1\leq i,j\leq\ell$) provide a realization of the affine Lie superalgebra 
$\widehat{\mathfrak{gl}}(m|n)$ inside the free superfermion algebra $F^1(\liea)$.
It is easy to see that if $\ph$ is the automorphism of $F^1(\liea)$ from Examples \ref{ex:symm1}, \ref{ex:skew1}, then $\ph$ preserves $\widehat{\mathfrak{gl}}(m|n)$.
In fact, $\ph$ acts as an inner automorphism of ${\mathfrak{gl}}(m|n)$; hence, the corresponding $\ph$-twisted modules are as in
\cite[Section 6.1]{Bak:2015}.

%%%%%%%%%%%%%%%%%%%%%%%%%%%%%%%%%%%%%%%%%%%%%%%%%
\section{Free fermions}\label{sec:ff}
%%%%%%%%%%%%%%%%%%%%%%%%%%%%%%%%%%%%%%%%%%%%%%%%%
In this section, we will compute explicitly the actions of $\caph$ and $L_0$ on $M_\ph(R)$ when $\liea$ is odd as in Examples \ref{ex:symm1}, \ref{ex:symm2}. Let $\{v_i\}$ be a basis for $\liea$ such that $(v_i|v_j) = \de_{i+j,d+1}$ $(1 \leq i,j \leq d)$, and $\ph$ acts as in Example \ref{ex:symm1} or \ref{ex:symm2}. Then the basis defined by $v^i = v_{d-i+1}$ is dual to $\{v_i\}$ with respect to $(\cdot|\cdot)$, and a conformal vector $\om$ is given by \eqref{fomega}.

%%%%%%%%%%%%%%%%%%%%%%%%%%%%%%%%%%%%%%%%%%%%%%%%%
 \subsection{The case of {E}xample \ref{ex:symm1}}

Assume that $\dim \liea = 2\ell$, and $\sigma$ and $\N$ act as in Example \ref{ex:symm1}. The logarithmic fields $Y(v_j,z)$ are given by \eqref{tfields1}. The Lie superalgebra $\caph$ is spanned by an even central element $K$ and odd elements $v_it^{m+\al_0}$, $v_{\ell+i}t^{m-\al_0}$ $(1 \leq i \leq \ell, \, m \in \ZZ)$. The nonzero brackets in $\caph$ are given by:
\begin{equation*}%\label{ffbrack1}
[v_i t^{m+\al_0},v_j t^{n-\al_0}] = \delta_{m,-n-1}\delta_{i+j,2\ell+1}K,
%[v_{\ell+i}t^{m-\al_0},v_jt^{n+\al_0}] & = \delta_{m,-n-1}\delta_{i+j,\ell+1}K, \nonumber
\end{equation*}
for $1 \leq i \leq \ell$, $\ell+1 \leq j \leq 2\ell$, $m,n \in \ZZ$. The elements of $\caph^-$ act as creation operators on $M_\ph(R)$. Throughout the rest of this section, we will represent them as anti-commuting variables as follows:
\begin{equation}\label{xi2}
v_it^{-m+\al_0}=\xi_{i,m}, \qquad v_j t^{-n-\al_0-1}=\xi_{j,n},
\end{equation}
for $v_i \in \liea_\al$, $v_j \in \liea_{-\al}$, and $m\geq 1$, $n\geq 0$. The precise triangular decomposition \eqref{catr} depends on whether $\al_0=-1/2$ or $\al_0 \in \CC_{-\frac{1}{2}}^+$ (cf.\ \eqref{cplus2}). 

Suppose first that $\al_0 \in \CC_{-\frac{1}{2}}^+$. Then $\caph^0 = \CC K$ and $R = \CC$. Thus
\begin{equation}\label{ffverma1}
M_\ph(R) \cong \bigwedge(\xi_{i,m+1},\xi_{\ell+i,m})_{1\leq i \leq \ell,\, m=0,1,2,\ldots}.
\end{equation}
The action of $\caph^+$ on $M_\ph(R)$ is given explicitly by
\begin{equation*}%\label{ffaction1}
v_i t^{m+\al_0} = \partial_{\xi_{2\ell-i+1,m}}, \qquad
v_{\ell+i}t^{n-\al_0-1} = \partial_{\xi_{\ell-i+1,n}},
\end{equation*}
where $1 \leq i \leq \ell$, $m\geq 0$, $n \geq 1$.  By Proposition \ref{prop:fvir}, the action of $L_0$ is
\begin{equation}\label{ffL1}
\begin{split}
L_0 =&\sum_{i=1}^\ell\sum_{m =1}^\infty \xi_{i,m}\bigg(\Big(m-\al_0-\frac{1}{2}\Big)\partial_{\xi_{i,m}}-(1-\delta_{i,1})\partial_{\xi_{i-1,m}}\bigg) \\
&\quad +\sum_{i=1}^\ell\sum_{m =0}^\infty  \xi_{\ell+i,m}\bigg(\Big(m+\al_0+\frac{1}{2}\Big)\partial_{\xi_{\ell+i,m}}+(1-\delta_{i,1})\partial_{\xi_{\ell+i-1,m}}\bigg) \\
&\quad +\frac{\ell}{2}\al_0^2I.
\end{split}
\end{equation}
Using \coref{fvirother}, we find that the action of $L_0^+$ is given by:
\begin{equation}\label{ffL1p}
\begin{split}
L_0^+ =&\sum_{i=1}^\ell\sum_{m =1}^\infty \xi_{i,m}\bigg((m-\al_0)\partial_{\xi_{i,m}}-(1-\delta_{i,1})\partial_{\xi_{i-1,m}}\bigg) \\
&\quad +\sum_{i=1}^\ell\sum_{m =0}^\infty  \xi_{\ell+i,m}\bigg((m+\al_0)\partial_{\xi_{\ell+i,m}}+(1-\delta_{i,1})\partial_{\xi_{\ell+i-1,m}}\bigg) \\
&\quad +\frac{\ell}{2}\al_0(\al_0-1)I,
\end{split}
\end{equation}
which corresponds to \eqref{SFL1} after relabeling of the variables. %and reordering two of the terms.

Now we consider the case when $\al_0 = -1/2$. Then 
\begin{equation}\label{fftriangle2}
\caph^0 = \Span\{v_it^{-1/2}\}_{1 \leq i \leq d}\oplus \CC K,
\end{equation}
where $d=2\ell$. We let
\begin{equation*}
R = \bigwedge(\xi_{\ell+i,0})_{1 \leq i \leq \ell}
\end{equation*}
with 
\begin{equation*}
v_{\ell+i}t^{-1/2} = \xi_{\ell+i,0}, \qquad v_{i}t^{-1/2} = \partial_{\xi_{2\ell-i+1,0}} \qquad (1 \leq i \leq \ell). 
\end{equation*}
Therefore, by \eqref{xi2}, $M_\ph(R)$ is again given by \eqref{ffverma1}.
The action of $\caph^+$ is given by 
\begin{equation*}
v_it^{m-1/2}  = \partial_{\xi_{2\ell-i+1,m}}, \qquad v_{\ell+i}t^{n-1/2} = \partial_{\xi_{\ell-i+1,n}},
\end{equation*}
for $1 \leq i \leq \ell$, $m \geq 0$, $n \geq 1$.  The actions of $L_0$ and $L_0^+$ are given by \eqref{ffL1} and \eqref{ffL1p} respectively, each with $\al_0=-\frac12$.

%%%%%%%%%%%%%%%%%%%%%%%%%%%%%%%%%%%%%%%%%%%%%%%%%
 \subsection{The case of {E}xample \ref{ex:symm2}}
Let $\liea$ be as in Example \ref{ex:symm2}. The logarithmic fields $Y(v_j,z)$ are the same as in \cite[Section 6.5]{Bak:2015}:
\begin{equation*}
Y(v_j,z) = \sum_{i=j}^{2\ell-1} \sum_{m\in\al_0+\ZZ} \frac{(-1)^{(i-j)(i+j-1)/2}}{(i-j)!} \, \zeta^{i-j} (v_i t^m) z^{-m-1}
\end{equation*}
for $1\leq j\leq 2\ell-1$.
The Lie superalgebra $\caph$ is spanned by an even central element $K$ and odd elements $v_it^{m+\al_0}$ $(1 \leq i \leq 2\ell-1, \, m \in\ZZ)$ where $\al_0 = -\frac12$ or $0$. The brackets in $\caph$ are given by
\begin{equation*}
[v_it^{m+\alpha_0},v_jt^{n-\alpha_0}] = \de_{m,-n-1}\de_{i+j,2\ell}K,
\end{equation*}
for $1 \leq i, j \leq 2\ell-1$, $m,n \in \ZZ$. 

We let the creation operators from $\caph^-$ act by the first equation of \eqref{xi2}. The triangular decomposition \eqref{catr} depends on whether $\alpha_0 = -\frac12$ or $\alpha_0=0$. We first consider the case when $\al_0 = 0$. Then $\caph^0 = \CC K$ and 
\begin{equation*}
M_\ph(R) \cong \bigwedge(\xi_{i,m})_{1 \leq i \leq 2\ell-1, \, m=1,2,3,\ldots},
\end{equation*}
where the action of $\caph^+$ on $M_\ph(R)$ is 
\begin{equation*}
v_it^{m}  = \partial_{\xi_{2\ell-i,m+1}}, \qquad 1 \leq i \leq 2 \ell - 1, \; m \geq 0.
\end{equation*}
The action of $L_0$ is
\begin{align*}
L_0 & = \sum_{i=1}^{2\ell-1}\sum_{m=1}^\infty\xi_{i,m} \bigg(\Big(m-\frac{1}{2}\Big)\partial_{\xi_{i,m}}+(-1)^{i+1}(1-\delta_{i,1})\partial_{\xi_{i-1,m}}\bigg).
\end{align*}

Now consider the case when $\al_0 = -1/2$. Then $\caph^0$ is given by \eqref{fftriangle2} with $d=2\ell-1$.
We let
\begin{equation*}
R = \bigwedge(\xi_{j,0})_{\ell \leq j \leq 2\ell-1}  \qquad \Bigl(\text{where } \xi_{\ell,0}^2 = \frac12 \Bigr),
\end{equation*}
with 
\begin{equation*}
v_{j}t^{-1/2}=\xi_{j,0}, \qquad v_{i}t^{-1/2}=\partial_{\xi_{2\ell-i,0}}  %\qquad (1 \leq i \leq \ell-1, \; \ell \leq j \leq 2\ell-1).
\end{equation*}
for $1 \leq i \leq \ell-1$ and $\ell \leq j \leq 2\ell-1$.
Therefore,
\begin{equation*}
M_\ph(R) \cong \bigwedge(\xi_{i,m+1},\xi_{j,m})_{1 \leq i \leq \ell-1, \, \ell\leq j \leq 2\ell-1 ,\, m =0,1,2,\ldots},
\end{equation*}
where the action of $\caph^+$ on $M_\ph(R)$ is given by
\begin{equation*}
v_{i}t^{m-1/2} = \partial_{\xi_{2\ell-i,m}}, \qquad 1 \leq i \leq 2\ell-1, \; m\geq 1.
\end{equation*}
The action of $L_0$ is
\begin{align*}
L_0 =\sum_{i=1}^{2\ell-1} & \sum_{m= 1}^\infty \xi_{i,m}\Big(m\partial_{\xi_{i,m}}-(-1)^i(1-\delta_{i,1})\partial_{\xi_{i-1,m}} \Big)\\
&+\sum_{i=\ell+2}^{2\ell-1}(-1)^{i+1}\xi_{i,0}\partial_{\xi_{i-1,0}}+(-1)^\ell \xi_{\ell+1,0}\xi_{\ell,0} +\frac{2\ell-1}{16}I.
\end{align*}

% If we take two pairs of  charged free fermions $\psi_1^{\pm}(z)$ and $\psi_2^{\pm}(z)$ (in our usual notation $v_i=\psi_i^+$ and $v_{4-i}=\psi_{2+i}^-$ for for $i=1,2$), then $\widehat{\lie{gl}}(2)$ is generated by the quadratic fields
% \begin{equation*}
% e_{ij} = \nop{\psi_i^+(z)\psi_j^-(z)}.
% \end{equation*}
% This follows from the computation
% \begin{equation*}
% \nop{\psi^+_i(z_1)\psi_j^-(z_1)}\nop{\psi^+_k(z_2)\psi_\ell^-(z_2)} \sim 
% \frac{\de_{jk}}{z_{12}}e_{i\ell}(z_2)-\frac{\de_{i\ell}}{z_{12}}e_{kj}(z_2) + \frac{\de_{i\ell}\de_{jk}}{z_{12}^2},
% \end{equation*}
% and the fact that in $\lie{gl}(2)$ we have 
% \begin{equation*}
% [e_{ij},e_{k\ell}]=\de_{j,k}e_{i\ell}-\de_{i\ell}e_{kj},
% \end{equation*}
% and
% \begin{equation*}
% (e_{ij}|e_{k\ell})=\tr (e_{ij}e_{k\ell})=\de_{jk}\de_{i\ell}.
% \end{equation*}
% If we take $\ph$ to be as in \exref{ex:symm1} with $\al_0=0$, then a short computation shows that $\N=\ad_{e_{21}}$ on $\lie{gl}(2)$. It follows that $\ph'=\ph\big|_{\lie{gl}(2)}$ is an inner automorphism of $\lie{gl}(2)$. So in this case, the $\ph$-twisted $\widehat{\lie{gl}}(2)$-module obtained via restriction is equivalent to the one obtained by taking the $\ph'$-twisted affinization $\widehat{\lie{gl}}(2)_{\ph'}$.
% \end{remark}

%%%%%%%%%%%%%%%%%%%%%%%%%%%%%%%%%%%%%%%%%%%%%%%%%
\section{Bosonic ghost system}\label{sec:bgs}
%%%%%%%%%%%%%%%%%%%%%%%%%%%%%%%%%%%%%%%%%%%%%%%%%
Now we will compute explicitly the actions of $(C_\liea)_\ph$ and $L_0$ on $M_\ph(R)$ when $\liea$ is even as in Examples \ref{ex:skew1} and \ref{ex:skew2}. 
Let $\{v_i\}_{1 \leq i \leq 2\ell}$ be a basis for $\liea$ such that $(v_i|v_j) = \de_{i+j,2\ell+1}$ $(1 \leq i \leq j \leq 2\ell)$, and $\ph$ acts as in Example \ref{ex:skew1} or \ref{ex:skew2}. Then the basis defined by $v^i = v_{2\ell-i+1}$, $v^{\ell+i} = -v_{\ell-i+1}$ $(1 \leq i \leq \ell)$ is dual to $\{v_i\}$ with respect to $(\cdot|\cdot)$, and a conformal vector is given by \eqref{fomega}.

%%%%%%%%%%%%%%%%%%%%%%%%%%%%%%%%%%%%%%%%%%%%%%%%%
 \subsection{The case of {E}xample \ref{ex:skew1}}\label{subsec:bg1}
Assume that $\si$ and $\N$ act as in Example \ref{ex:skew1}. The logarithmic fields $Y(v_j,z)$ are given by \eqref{tfields1}. The Lie algebra $\caph$ is spanned by a central element $K$ and elements $v_it^{m+\al_0}$, $v_{\ell+i}t^{m-\al_0}$ $(1 \leq i \leq \ell, \, m \in \ZZ)$. The nonzero brackets in $\caph$ are given by 
\begin{equation}\label{bgbrack1}
[v_it^{m+\al_0},v_jt^{n-\al_0}] = \de_{m,-n-1}\de_{i+j,2\ell+1}K,
%[v_{\ell+i}t^{m-\al_0},v_jt^{n+\al_0}]& = -\de_{m,-n-1}\de_{i+j,\ell+1}K,\nonumber
\end{equation}
for $1 \leq i \leq \ell$, $\ell+1 \leq j \leq 2 \ell$, $m,n \in \ZZ$. The elements of $\caph^-$ act as creation operators on $M_\ph(R)$. Throughout the rest of this section, we will represent them as commuting variables as follows
\begin{equation}\label{commvar}
v_it^{-m+\al_0}=x_{i,m}, \qquad v_j t^{-n-\al_0-1} = x_{j,n}, 
\end{equation}
for $1 \leq i \leq \ell$, $\ell+1 \leq j \leq 2\ell$
and $m\geq 1$, $n \geq 0$. Again, the precise triangular decomposition \eqref{catr} depends on whether $\al_0 \in \CC_{-\frac{1}{2}}^+$ or $\al_0 = -\frac12$. 

Consider first the case when $\al_0 \in \CC_{-\frac{1}{2}}^+$. Then $\caph^0 = \CC K$, $R=\CC$, and 
\begin{equation}\label{bgverma1}
M_\ph(R) = \CC[x_{i,m+1},x_{\ell+i,m}]_{1 \leq i \leq  \ell ,\, m = 0,1,2,\ldots}.
\end{equation}
The action of $\caph^+$ on $M_\ph(R)$ is given explicitly by 
\begin{equation}\label{bgaction1}
v_it^{m+\al_0}=\partial_{x_{2\ell-i+1,m}},\qquad
 v_{\ell+i}t^{n-\al_0-1}  = -\partial_{x_{\ell-i+1,n}},
\end{equation}
for $1 \leq i \leq \ell$, $m \geq 0$, $n \geq 1$.
The action of $L_0$ is 
\begin{equation}\label{bgL1}
\begin{split}
L_0 =&\sum_{i=1}^\ell\sum_{m =1}^\infty x_{i,m}\bigg(\Big(m-\al_0-\frac{1}{2}\Big)\partial_{x_{i,m}}-(1-\delta_{i,1})\partial_{x_{i-1,m}}\bigg) \\
&+\sum_{i=1}^\ell\sum_{m =0}^\infty  x_{\ell+i,m}\bigg(\Big(m+\al_0+\frac{1}{2}\Big)\partial_{x_{\ell+i,m}}+(1-\delta_{i,1})\partial_{x_{\ell+i-1,m}}\bigg) \\
&-\frac{\ell}{2}\al_0^2I. 
\end{split}
\end{equation}
By \coref{fvirother}, the action of $L_0^+$ is given by:
\begin{equation}\label{bgL1p}
\begin{split}
L_0^+ =&\sum_{i=1}^\ell\sum_{m =1}^\infty x_{i,m}\bigg((m-\al_0)\partial_{x_{i,m}}-(1-\delta_{i,1})\partial_{x_{i-1,m}}\bigg) \\
& +\sum_{i=1}^\ell\sum_{m =0}^\infty  x_{\ell+i,m}\bigg((m+\al_0)\partial_{x_{\ell+i,m}}+(1-\delta_{i,1})\partial_{x_{\ell+i-1,m}}\bigg) \\
& -\frac{\ell}{2}\al_0(\al_0-1)I,
\end{split}
\end{equation}
which agrees with \cite[Lemma 6.8]{Bak:2015} up to relabeling of the variables.

Now consider the case when $\al_0=-\frac{1}{2}$. Then $\caph^0$ is given by \eqref{fftriangle2} with $d=2\ell$.
We let
\begin{equation}\label{bgR2}
R = \CC[x_{\ell+i,0}]_{1 \leq i \leq \ell},
\end{equation}
with 
\begin{equation*}
v_{\ell+i}t^{-1/2} = x_{\ell+i,0}, \qquad v_{i}t^{-1/2}=\partial_{x_{2\ell-i+1,0}} \qquad (1 \leq i \leq \ell).
\end{equation*}
Then $M_\ph(R)$ is given by \eqref{bgverma1},
% \begin{equation}\label{bgverma2}
% M_\ph(R) \cong \CC[x_{i,m},x_{\ell+i,m+1}]_{1 \leq i \leq \ell, m = 0,1,2,\ldots},
% \end{equation}
where the action of $\caph^+$ is given by 
\begin{align}\label{bgaction2}
v_it^{m-1/2}&=\partial_{x_{2\ell-i+1,m}}, &v_{\ell+i}t^{m+1/2}&=-\partial_{x_{\ell-i+1,m+1}},
\end{align}
for $1 \leq i \leq \ell$ and $m\geq 1$. The actions of $L_0$ and $L_0^+$ are given by \eqref{bgL1} and \eqref{bgL1p} respectively, each with $\al_0=-\frac12$.

%%%%%%%%%%%%%%%%%%%%%%%%%%%%%%%%%%%%%%%%%%%%%%%%%
 \subsection{The case of {E}xample \ref{ex:skew2}}
Let $\liea$ be as in Example \ref{ex:skew2}. The logarithmic fields are given by \eqref{tfields2}. The Lie algebra $\caph$ is spanned by a central element $K$ and elements $v_it^{m+\al_0}$ $(1 \leq i \leq 2 \ell)$ where $\al_0 = 0$ or $-1/2$. The brackets in $\caph$ are given by \eqref{bgbrack1}. 
We let the creation operators from $\caph^-$ act on $M_\ph(R)$ by \eqref{commvar}. As before, the triangular decomposition depends on whether $\al_0=0$ or $-1/2$.

We first consider the case when $\al_0=0$. Then $(C_\liea )_\ph^0 = \CC K$, $R=\CC$, $M_\ph(R)$ is given by \eqref{bgverma1}, and the action of $(C_\liea )_\ph^+$ is given by \eqref{bgaction1} with $\al_0=0$. The action of $L_0$ is 
\begin{align*}
L_0 =&\sum_{i=1}^\ell\sum_{m =1}^\infty x_{i,m}\bigg(\Big(m-\frac{1}{2}\Big)\partial_{x_{i,m}}-(1-\delta_{i,1})\partial_{x_{i-1,m}}\bigg) \\
&\quad +\sum_{i=1}^\ell\sum_{m =1}^\infty  x_{\ell+i,m}\bigg(\Big(m+\frac{1}{2}\Big)\partial_{x_{\ell+i,m}}+(1-2\delta_{i,1})\partial_{x_{\ell+i-1,m}}\bigg)\\
&\quad + \frac{1}{2}\sum_{i=1}^\ell x_{\ell+i,0}\partial_{x_{\ell+i,0}}+ \sum_{i=2}^\ell x_{\ell+i,0}\partial_{x_{\ell+i-1,0}}.
\end{align*}

Now assume $\al_0 = -1/2$. Then $(C_\liea)_\ph^0$, $R$, and  $M_\ph(R)$ are the same as in the case when $\al_0=-1/2$ in Section \ref{subsec:bg1}. The action of $L_0$ is

\begin{align*}
L_0 =\sum_{i=1}^\ell & \sum_{m =1}^\infty x_{i,m}\bigg(m\partial_{x_{i,m}}-(1-\delta_{i,1})\partial_{x_{i-1,m}}\bigg) \\
& +\sum_{i=1}^\ell\sum_{m =1}^\infty  x_{\ell+i,m}\bigg(m\partial_{x_{\ell+i,m}}+(1-2\delta_{i,1})\partial_{x_{\ell+i-1,m}}\bigg)\\
&+\sum_{i=2}^\ell  x_{\ell+i,0} \partial_{x_{\ell+i-1,0}}
+ \frac12 x_{\ell+1,0}^2-\frac{\ell}{8} I. 
\end{align*}

%%%%%%%%%%%%%%%%%%%%%%%%%%%%%%%%%%%%%%%%%%%%%%%%%
%%%%%%%%%%%%%%%%%%%%%%%%%%%%%%%%%%%%%%%%%%%%%%%%%
%\section{Conclusion}
%%%%%%%%%%%%%%%%%%%%%%%%%%%%%%%%%%%%%%%%%%%%%%%%%
%%%%%%%%%%%%%%%%%%%%%%%%%%%%%%%%%%%%%%%%%%%%%%%%%

%%%%%%%%%%%%%%%%%%%%%%%%%%%%%%%%%%%
% \bibliography{freefields}{}
% \bibliographystyle{plain}
%%%%%%%%%%%%%%%%%%%%%%%%%%%%%%%%%%%
\bibliographystyle{amsalpha}

\end{document}